\documentclass[11pt, reqno]{amsart}

\usepackage{amsmath,amssymb,amsthm, epsfig}
\usepackage{hyperref}
\usepackage{amsmath}
\usepackage{amssymb}
\usepackage{color}
\usepackage{ulem}
\usepackage{epsfig}
\usepackage[mathscr]{eucal}
\usepackage[latin1]{inputenc}

\newtheorem{theorem}{Theorem}
\newtheorem{definition}{Definition}

\newtheorem{lemma}{Lemma}
\newtheorem{proposition}{Proposition}
\newtheorem{corollary}{Corollary}
\newtheorem{remark}{Remark}

\date{}
\numberwithin{equation}{section}
\numberwithin{theorem}{section}
\numberwithin{lemma}{section}
\numberwithin{corollary}{section}
\numberwithin{remark}{section} 
\numberwithin{proposition}{section}
\numberwithin{definition}{section}

\def \Div {\mathrm{div}}

\def \R {\mathbb{R}}

\def \loc {\mathrm{loc}}

\begin{document}

\title[The fractional Laplacian: a primer]{The fractional Laplacian: a primer}

\author[R. Teymurazyan]{Rafayel Teymurazyan}
\address{King Abdullah University of Science and Technology (KAUST), Computer, Electrical and Mathematical Sciences and Engineering Division (CEMSE), Thuwal 23955-6900, Saudi Arabia and University of Coimbra, CMUC, Department of Mathematics, Largo D. Dinis, 3000-143 Coimbra, Portugal}{} 
\email{rafayel.teymurazyan@kaust.edu.sa} 

\begin{abstract}
	
	In this note we give a glimpse of the fractional Laplacian. In particular, we bring several definitions of this non-local operator and series of proofs of its properties. It is structured in a way as to show that several of those properties are natural extensions of their local counterparts, with some key differences.
	
\bigskip

\noindent \textbf{Keywords:} Fractional Laplacian; comparison principle; Harnack inequality; Liouville theorem; approximation.

\bigskip

\noindent \textbf{MSC 2020:} 35R11, 26A33, 47G30.
\end{abstract}

\maketitle

\section{Introduction}\label{s1}

During the last decades the study of non-local equations was boosted by large range of applications in financial mathematics (as a pricing model for American options, \cite{S07}), optimal design problems, \cite{TT15}, competitive stochastic games, \cite{CS09,CTU20}, population dynamics, combustion processes, catalysis process, bio-technologies, chemical engineering, and other areas. This note is a primer to a classical example of a non-local operator - the fractional Laplacian. Over the years several advanced and comprehensive notes and books have been written on the subject, such as \cite{AV19,B16,BV16,DPV12,G19} among others. This primer is intended for those students and young researchers who are already acquainted with the classical Laplace operator and want to get a brief sense of the fractional Laplacian. The latter being a lot like classical Laplacian, at the same time is also quite different. As we will see, the fractional Laplacian can be defined in various ways, and like in case of the classical Laplacian, it satisfies to a mean value property, maximum principle, Harnack inequality, Liouville theorem, and so forth. Moreover, there are Poisson formula and Green function available, and fractional harmonic functions, like classical harmonic functions, are $C^\infty$. Of course, the non-local nature of the fractional Laplacian dictates certain modifications in those results. However, once the reader gets a glimpse of the story ``behind the scene'', these modifications seem quite natural. Obviously, there are also striking differences between these operators. In this note we emphasize those too.

The Laplacian,
$$
\Delta u:=\Div(\nabla u),
$$
is a classical example of a local operator. It arises naturally, when for example, looking at a Brownian motion originated in a bounded domain (with a smooth boundary): the expected value of a function, when the motion hits the boundary for the first time, solves a Dirichlet problem for the Laplace operator. In other words, the unique solution of the problem
\begin{equation*}
	\begin{cases}
		\Delta u=0 &\textrm{ in }\,\,\Omega,\\
		u=f & \textrm{ on }\,\,\partial\Omega,
	\end{cases}
\end{equation*}
where $\Omega$ is a bounded domain with a smooth boundary, and $f\in C(\partial\Omega)$, is given by 
$$
u(x)=\mathbb{E}\left(f(X_\tau)\right).
$$
Here $x\in\Omega$ is the point where the Brownian motion originated, $X_\tau$ is where it hits the boundary for the first time ($\tau$ is the stopping time), and $\mathbb{E}$ is the expected value of the process. If the process tends to move in certain directions more than in other directions, then we deal with equations with coefficients. These kind of models arise in electromagnetism, fluid dynamics, thermodynamics, etc. (see, for example, \cite{KS91}). Thus, continuous processes lead to a local problem. Jump processes, on the other hand, lead to non-local problems, \cite{BC83,CS09,CTU20,KS91}. If in the example above instead of a continuous process one deals with a jump process, then we end up solving the non-local Dirichlet problem. More precisely, for a purely jump L\'evy process, originated in a bounded domain, the expected value of the function at the first exit point solves the non-local Dirichlet problem, i.e., the unique solution of the problem
\begin{equation*}
	\begin{cases}
		(-\Delta)^su=0 &\textrm{ in }\,\,\Omega,\\
		u=f & \textrm{ in }\,\,\R^n\setminus\Omega,
	\end{cases}
\end{equation*}
is given by
$$
u(x)=\mathbb{E}\left(f(X_\tau)\right).
$$
As before, $x\in\Omega$ is the point where the jump process originated, $f\in C(\R^n\setminus\Omega)$ and $X_\tau\in\R^n\setminus\Omega$ is the first exit point. The operator $(-\Delta)^s$ is the fractional Laplacian and for $s\in(0,1)$ is defined by
\begin{equation*}
	\begin{split}
(-\Delta)^su(x):&=c_{n,s}\,\textrm{P.V.}\int_{\R^n}\frac{u(x)-u(y)}{|x-y|^{n+2s}}\,dy\\
&=c_{n,s}\lim_{\varepsilon\to0+}\int_{\R^n\setminus B_\varepsilon(x)}\frac{u(x)-u(y)}{|x-y|^{n+2s}}\,dy,
	\end{split}
\end{equation*}
where $c_{n,s}$ is a normalization constant depending only on $n$ and $s$. Here P.V. indicates that the integral should be understood in the ``principle value sense'' (defined by the last equality). Observe that unlike the problem driven from a continuous process, the boundary in the model obtained for a jump process is substituted by the whole complement of the set $\Omega$. The idea behind this is that when jumping out of $\Omega$ the process can end up at any point in $\Omega^c$. That is, the complement of the domain in the non-local setting plays the role of the boundary in the local setting. This fact has its reflection on the modifications of some basic properties, as we will see later.

In the local setting, to check whether a partial differential equation holds at a particular point, one needs to know only the values of the function in an arbitrarily small neighborhood of that point, whereas in the non-local setting it is the opposite: in order to check whether a non-local equation holds at a point, one needs information about the values of the function far away from that point. Therefore, when considering long-range integration, non-local models become more accurate. In other words, unlike local versions of problems, which can feel changes only on the boundary of the substance, non-local models become sensitive to changes that occur faraway. The following simple example shows the effect of non-locality. If $0\le u\le1$ is such that $u\in C_0^\infty(B_2)$ and $u\equiv1$ in $B_1$, then for any $x\in\R^n\setminus B_4$ one has $\Delta u(x)=0$, while
\begin{equation*}
	\begin{split}
		-(-\Delta)^su(x)&=c_{n,s}\,\textrm{P.V.}\int_{\R^n}\frac{u(y)-u(x)}{|x-y|^{n+2s}}\,dy=\int_{B_2}\frac{u(y)}{|x-y|^{n+2s}}\,dy\\
		&\ge\int_{B_1}\frac{dy}{\left(|x|+1\right)^{n+2s}}\ge C|x|^{-n-2s},
	\end{split}
\end{equation*}
for a constant $C>0$. In fact, $|(-\Delta)^su(x)|\le C|x|^{-n-2s}$ (see \cite[Appendix B]{AV19}, for example). As we will see below, the non-local nature of the fractional Laplacian endows somewhat surprising behavior for solutions of equations driven by it, a remarkable example of which is the fact that any (smooth) function is fractional harmonic up to a small error, Theorem \ref{approximations}.

This note is organized as follows. After introducing some notations, in Section \ref{s2} we bring several definitions of the fractional Laplacian. Yet another definition of this operator is given in Section \ref{extensionsection}, where also its fundamental solution and several properties are discussed. Some elementary properties are presented in Section \ref{s3}. In Section \ref{s4}, the mean value property is proved. Section \ref{s5} is dedicated to the maximum principle. Section \ref{s6} is devoted to the Harnack inequality. Liouville theorem in the fractional setting is proved in Section \ref{s7}, followed by Schauder type estimates in Section \ref{s8}. Section \ref{s10} concerns Green's function for the ball. In Section \ref{s11} it is shown that fractional harmonic functions are locally $C^\infty$. Finally, in Section \ref{s12}, we see that all functions are fractional harmonic up to a small error.

\section*{Notations}

$\Omega\subset\R^n$ is a bounded domain;

$D_iu:=D_{x_i}u:=\frac{\partial u}{\partial x_i}$;

$Du:=(D_1u,D_2u,\ldots,D_nu)$;

$D_\nu u:=\frac{\partial u}{\partial\nu}=Du\cdot\nu$;

$x_+:=\max\{x,0\}$.

For a multi-index $\gamma=(\gamma_1,\gamma_2,\ldots,\gamma_n)$, we use $|\gamma|:=\gamma_1+\gamma_2+\ldots+\gamma_n$.

For $\alpha\in(0,1]$ and $k\in\mathbb{N}$, the H\"older semi-norm is defined as follows:
$$
[u]_{C^{0,\alpha}(\Omega)}:=\sup_{x\neq y}\frac{|u(x)-u(y)|}{|x-y|^\alpha},
$$
$$
[u]_{C^{k,\alpha}(\Omega)}:=\max_{|\gamma|=k}[D^\gamma u]_{C^{0,\alpha}(\Omega)},
$$

where $D^\gamma u:=\partial_{x_1}^{\gamma_1}\ldots\partial_{x_n}^{\gamma_n}u$. 

The H\"older space $C^{k,\alpha}(\Omega)$ consists of all functions $u\in C^k(\Omega)$ for which
$$
\|u\|_{C^{k,\alpha}(\Omega)}:=\sum_{|\gamma|\le k}\|D^\gamma u\|_{C(\Omega)}+\sum_{|\gamma|=k}[D^\gamma u]_{C^{0,\alpha}(\Omega)}<\infty.
$$

$C^\alpha:=C^{0,\alpha}$, if $\alpha\in(0,1]$ and $C^\alpha:=C^{1,\alpha-1}$, if $\alpha\in(1,2]$, and similarly,

$C^{k+\alpha}:=C^{k,\alpha}$, if $\alpha\in(0,1]$ and $C^{k+\alpha}:=C^{k+1,\alpha-1}$, if $\alpha\in(k,k+1]$.\\

We use $\mathcal{S}$ for the Schwartz space of rapidly decreasing $C^\infty$ functions in $\R^n$. More precisely, 
$$
\mathcal{S}:=\left\{u\in C^\infty(\R^n);\,\,\sup_{x\in\R^n}|x^\beta D^\alpha u(x)|<\infty, \forall\alpha,\beta\in\mathbb{N}_0^n\right\}.
$$

For $s\in(0,1)$, set
\begin{equation}\label{L1s}
L_s^1(\R^n):=\left\{u\in L_{\loc}^1(\R^n);\,\,\int_{\R^n}\frac{|u(y)|}{1+|y|^{n+2s}}\,dy<+\infty\right\}.
\end{equation}
Also $B_r(x_0)$ is the ball of radius $r$ centered at $x_0$, and $B_r:=B_r(0)$.

\section{Several definitions of the fractional Laplacian}\label{s2}
In this section we bring five definitions of the fractional Laplacian (one more definition is given in the next section). These definitions are all equivalent once $u\in\mathcal{S}$, \cite{K17}. There are several other definitions of the fractional Laplacian. We refer the interested reader to, for example, \cite{G19,K17,S19}. 

\subsection{As a singular integral}
For $s\in(0,1)$ and $u\in\mathcal{S}$, the \textit{fractional Laplacian} of $u$ is defined as
	\begin{equation}\label{1.1}
		(-\Delta)^su(x):=c_{n,s}\,\textrm{P.V.}\int_{\R^n}\frac{u(x)-u(y)}{|x-y|^{n+2s}}\,dy.
	\end{equation}
Here 
\begin{equation}\label{1.2}
	c_{n,s}:=\int_{\R^n}\frac{1-\cos \zeta_1}{|\zeta|^{n+2s}}\,d\zeta
\end{equation}
is a normalization constant depending only on $n$ and $s$ (see the proof of Proposition \ref{p2.1} below). The integral in \eqref{1.1} is absolutely convergent when $0<s<1/2$. Indeed,
\begin{equation*}
	\begin{split}
		\int_{\R^n}\frac{|u(x)-u(y)|}{|x-y|^{n+2s}}\,dy&\le C\int_{B_r}\frac{|x-y|}{|x-y|^{n+2s}}\,dy+\|u\|_{L^\infty(\R^n)}\int_{\R^n\setminus B_r}\frac{dy}{|x-y|^{n+2s}}\\
		&\le C\left[\int_{B_r}\frac{dy}{|x-y|^{n+2s-1}}+\int_{\R^n\setminus B_r}\frac{dy}{|x-y|^{n+2s}}\right]\\
		&=C\left[\int_0^r\frac{dt}{|t|^{2s}}+\int_r^{+\infty}\frac{dt}{|t|^{2s+1}}\right]<\infty,
	\end{split}
\end{equation*}
where the constant $C>0$ depends only on $\|u\|_{L^\infty(\R^n)}$, $\|Du\|_{L^\infty(\R^n)}$ and $n$. As for $1/2\le s<1$, the integral in \eqref{1.1} is understood in the ``Principle Value'' sense, i.e.,
$$
(-\Delta)^su(x)=c_{n,s}\lim_{\varepsilon\to0+}\int_{\R^n\setminus B_\varepsilon(x)}\frac{u(x)-u(y)}{|x-y|^{n+2s}}\,dy.
$$
For $s\in(0,1)$, the constant $c_{n,s}$ in \eqref{1.1} does not play any essential role on the properties of the fractional Laplacian. Its role is important only in the limits as $s\to0^+$ and $s\to1^-$ (for the asymptotic of this constant, as $s\to0^+$ and $s\to1^-$ see \cite[Section 4]{DPV12}). 

Observe that although in \eqref{1.1}, the fractional Laplacian was defined for $u\in\mathcal{S}$, however, the integral is well defined for less regular functions. In fact, the assumption on $u$ at infinity can be weakened by assuming $u\in L_s^1(\R^n)$, where
$L_s^1(\R^n)$ is defined by \eqref{L1s}. This can be checked using approximation with Schwartz functions (for details we refer the reader to \cite[Proposition 2.4]{S07}). Furthermore, the $C^\infty$ regularity requirement on $u$ can be relaxed as well by asking just $u\in C^{2s+\varepsilon}$ in a neighborhood of $x\in\R^n$, for $\varepsilon>0$ small. Indeed, for $s\in(0,\frac{1}{2})$ and $2s+\varepsilon\le1$ and $r>0$ small we have
$$
\int_{B_r(x)}\frac{u(x)-u(y)}{|x-y|^{n+2s}}\,dz\le[u]_{C^{2s+\varepsilon}(B_r(x))}\int_{B_r(x)}\frac{|x-y|^{2s+\varepsilon}}{|x-y|^{n+2s}}\,dy<\infty,
$$
hence the fractional Laplacian is well defined at $x$ by \eqref{1.1}. For $s\in[\frac{1}{2},1)$ still $u\in C^{2s+\varepsilon}=C^{1,2s+\varepsilon-1}$ would suffice.

\subsection{Removing singularity} 
Note that, in general, the right hand side of \eqref{1.1} is not well defined, as the integral may have singularity near $x$. Also, one would like to get rid of the P.V. in the definition. As the kernel in \eqref{1.1} is symmetric, a simple change of variable, $z:=y-x$, yields
\begin{equation}\label{2.1}
	\begin{split}
		(-\Delta)^su(x)&=c_{n,s}\,\textrm{P.V.}\int_{\R^n}\frac{u(x)-u(y)}{|x-y|^{n+2s}}\,dy\\
		&=c_{n,s}\,\textrm{P.V.}\int_{\R^n}\frac{u(x)-u(x+z)}{|z|^{n+2s}}\,dz\\
		&=c_{n,s}\,\textrm{P.V.}\int_{\R^n}\frac{u(x)-u(x-z)}{|z|^{n+2s}}\,dz.
	\end{split}
\end{equation}
This leads to the following definition of the fractional Laplacian.
\begin{definition}\label{d2.1}
	For $s\in(0,1)$ and $u\in\mathcal{S}$, the fractional Laplacian of $u$ is defined by
	\begin{equation}\label{2.2}
		(-\Delta)^su(x):=\frac{c_{n,s}}{2}\int_{\R^n}\frac{2u(x)-u(x+y)-u(x-y)}{|y|^{n+2s}}\,dy,
	\end{equation}
where $c_{n,s}$ is the constant defined by \eqref{1.2}.
\end{definition}
Indeed, from \eqref{2.1} one gets
$$
(-\Delta)^su(x)=\frac{c_{n,s}}{2}\,\textrm{P.V.}\int_{\R^n}\frac{2u(x)-u(x+y)-u(x-y)}{|y|^{n+2s}}\,dy.
$$
This representation of the fractional Laplacian removes the singularity at the origin, as the second order Taylor expansion gives
$$
\frac{2u(x)-u(x+y)-u(x-y)}{|y|^{n+2s}}\le\frac{\|D^2u\|_{L^\infty(\R^n)}}{|y|^{n+2s-2}},
$$
which is integrable near zero. Thus, we can remove P.V. in the previous equality and get \eqref{2.2}. 
\begin{remark}\label{r2.1}	
	As a consequence, $(-\Delta)^su$ is in fact well defined by \eqref{2.2} for any $u\in C^2(\R^n)\cap L^\infty(\R^n)$. In that sense any constant, although not being in $\mathcal{S}$ (unless identically zero), is fractional harmonic.
\end{remark}

\subsection{As a distribution} 
In $L_s^1(\R^n)$ fractional Laplacian can be defined as a distribution by
\begin{equation}\label{distribution}
	\langle(-\Delta)^su,\varphi\rangle:=\int_{\R^n}u(x)(-\Delta)^s\varphi(x)\,dx,\,\,\,\forall\varphi\in C_0^\infty(\R^n).
\end{equation}
In other words, for the definition of the fractional Laplacian to make sense, it is enough to assume that $u$ is locally integrable and has a suitable growth control at infinity. 

\subsection{As a generator of a L\'evy process}
Fractional Laplacian can be defined also as a generator of $2s$-stable L\'evy process, \cite{A09}. More precisely, if $X_t$ is the isotopic $2s$-stable L\'evy process starting at $0$, then for a smooth function $u$
$$
(-\Delta)^su(x)=\lim_{t\to0^+}\frac{1}{t}\mathbb{E}\left[u(x)-u(x+X_t)\right].
$$

\subsection{As a Fourier transform}
The fractional Laplacian is a pseudo-differential operator, as suggests the following proposition. It is here that the choice of the constant $c_{n,s}$ becomes evident.

\begin{proposition}\label{p2.1}
	If $s\in(0,1)$ and $u\in\mathcal{S}$, then
	\begin{equation}\label{Fourier}
		(-\Delta)^su(x)=\mathcal{F}^{-1}\left(\left(2\pi|\xi|\right)^{2s}\hat{u}(\xi)\right),\,\,\,\forall\xi\in\R^n,
	\end{equation}
where $\mathcal{F}=\hat{u}$ is the Fourier transform, i.e.
$$
\mathcal{F}u(\xi):=\hat{u}(\xi):=\int_{\R^n}u(x)e^{-2\pi i\xi\cdot x}\,dx.
$$
\end{proposition}
\begin{proof}
	This follows by applying Fourier transform in \eqref{2.2} and using Fubini theorem. Indeed, as observed above, \eqref{2.2} removes singularity at the origin, and hence, the integrant is in $L^1$. Using Fubini theorem, we then exchange the integral in $y$ with the Fourier transform in $x$. Thus, if $\xi$ is the frequency variable, from \eqref{2.2} one has
\begin{equation*}
	\begin{split}
		\mathcal{F}\left((-\Delta)^su(x)\right)(\xi)&=\frac{c_{n,s}}{2}\int_{\R^n}\frac{\mathcal{F}\left(2u(x)-u(x+y)-u(x-y)\right)}{|y|^{n+2s}}\,dy\\
		&=\frac{c_{n,s}}{2}\,\int_{\R^n}\hat{u}(\xi)\frac{2-e^{2\pi i\xi\cdot y}-e^{-2\pi i\xi\cdot y}}{|y|^{n+2s}}\,dy\\
		&=c_{n,s}\,\hat{u}(\xi)\int_{\R^n}\frac{1-\cos(2\pi\xi\cdot y)}{|y|^{n+2s}}\,dy.
	\end{split}
\end{equation*}
	Therefore, to see \eqref{Fourier}, it remains to check
	\begin{equation}\label{2.6}
		c_{n,s}\,\int_{\R^n}\frac{1-\cos(2\pi\xi\cdot y)}{|y|^{n+2s}}\,dy=\left(2\pi|\xi|\right)^{2s}.
	\end{equation}
Set
$$
I(\xi):=\int_{\R^n}\frac{1-\cos(2\pi\xi\cdot y)}{|y|^{n+2s}}\,dy.
$$
where we used the change of variable $z:=|\xi|y$ (still labeling the new variable with $y$).
If $R$ is some rotation that takes $e_1=(1,0,0,\ldots,0)$ to $\xi/|\xi|$, i.e., $Re_1=\xi/|\xi|$, then
\begin{equation*}
	\begin{split}
		I(\xi)&=|\xi|^{2s}\int_{\R^n}\frac{1-\cos\left(2\pi Re_1\cdot y\right)}{|y|^{n+2s}}\,dy\\
		&=|\xi|^{2s}\int_{\R^n}\frac{1-\cos\left(2\pi R^Ty\cdot e_1\right)}{|y|^{n+2s}}\,dy\quad(z:=R^Ty)\\
		&=|\xi|^{2s}\int_{\R^n}\frac{1-\cos\left(2\pi z_1\right)}{|z|^{n+2s}}\,dz\quad(\zeta:=2\pi z)\\
		&=\left(2\pi|\xi|\right)^{2s}\int_{\R^n}\frac{1-\cos\zeta_1}{|\zeta|^{n+2s}}\,d\zeta,
	\end{split}
\end{equation*}
which confirms \eqref{2.6}, since $c_{n,s}$ is defined by \eqref{1.2}. Here $z_1$ and $\zeta_1$ are the first coordinate of the vector $z$ and $\zeta$ respectively.

To be correct, one needs to make sure that the constant $c_{n,s}$ is a finite number. This is indeed the case, as inside the ball $B_1$, using the Taylor expansion of the cosine function, one estimates
$$
\int_{B_1}\frac{|1-\cos\zeta_1|}{|\zeta|^{n+2s}}\,d\zeta\le\int_{B_1}\frac{|\zeta_1|^2}{|\zeta|^{n+2s}}\,d\zeta\le\int_{B_1}\frac{1}{|\zeta|^{n+2s-2}}\,d\zeta<\infty,
$$
and outside of $B_1$ we have
$$
\int_{\R^n\setminus B_1}\frac{|1-\cos\zeta_1|}{|\zeta|^{n+2s}}\,d\zeta\le\int_{\R^n\setminus B_1}\frac{2}{|\zeta|^{n+2s}}\,d\zeta<\infty.
$$
\end{proof}
\begin{remark}\label{r2.2}
	The constant $c_{n,s}$ defined by \eqref{1.2}, can be written in terms of the Gamma function in the following way
	\begin{equation*}\label{constant}
		c_{n,s}=\frac{s4^s\Gamma\left(\frac{n+2s}{s}\right)}{\pi^{\frac{n}{2}}\Gamma(1-s)},
	\end{equation*}
where 
$$
\Gamma(r):=\int_0^\infty t^{r-1}e^{-t}\,dt,\,\,\,r>0.
$$
We refer the reader to \cite[Propositions 5.6 and 5.1]{G19} and \cite[Lemma 2.3]{BV16}, where the calculations are carried out. 
\end{remark}
This last definition of the fractional Laplacian can be used to prove the following integration by parts formula and construct a non-trivial example of a fractional harmonic function. Namely, if $u$, $v\in\mathcal{S}$, then
\begin{equation}\label{integrationbyparts}
	\int_{\R^n}(-\Delta)^su(x)v(x)\,dx=\int_{\R^n}u(x)(-\Delta)^sv(x)\,dx.
\end{equation}
When $s=1$, \eqref{integrationbyparts} is just integration by parts. For $s\in(0,1)$ it follows from \eqref{Fourier}, \cite[Lemma 5.4]{G19}.

As commented above, Remark \ref{r2.1}, constant functions are fractional $s$-harmonic. Below we bring another example of an $s$-harmonic function.
\begin{theorem}\label{t5.1}
	The function $u(x):=x^s_+$ is $s$-harmonic in the upper half space. More precisely,
	\begin{equation*}
		(-\Delta)^su(x)=
		\begin{cases}
			0, & x>0,\\
			-C|x|^{-s}, & x<0,
		\end{cases}
	\end{equation*}
	where $C>0$ is a constant depending only on $s$.
\end{theorem}
\begin{proof}
	There are several proofs of this fact, \cite[Section A.1 and Theorem 2.4.1]{BV16}. It can be shown by direct calculations making use of the definition of the fractional Laplacian via Fourier transform, \eqref{Fourier}. For the probabilistic intuition behind this, we refer the reader to \cite[Section 2.4]{BV16}.
\end{proof}

\section{An extension argument and beyond}\label{extensionsection}
Another definition of the fractional Laplacian can be given using the celebrated Caffarelli-Silvestre extension problem, \cite{CS07} (for the argument in probabilistic terms see \cite{MO69}). The construction of the extension hints a good candidate for the fundamental solution of the fractional Laplacian. More precisely, for a function $u:\R^n\to\R$, consider its extension to the upper half space, i.e., $v:\R^n\times[0,+\infty)\to\R$ such that it satisfies the following equation
\begin{equation}\label{2.8}
	\Delta_x v+\frac{1-2s}{y}v_y+v_{yy}=0,
\end{equation}
\begin{equation}\label{2.9}
	v(x,0)=u(x),
\end{equation}
where $v_y=\frac{\partial v}{\partial y}$. Note that \eqref{2.8} can be written as
\begin{equation}\label{2.10}
\Div\left(y^{1-2s}\nabla v\right)=0,
\end{equation}
which is the Euler-Lagrange equation of the functional
$$
\int_{y>0}y^{1-2s}|\nabla v|^2\,dx\,dy.
$$
To understand the intuition behind \eqref{2.8}, suppose for a moment, that $\tau:=1-2s$ is a non-negative integer and $v(x,y):\R^n\times\R^{1+\tau}\to\R$ is radially symmetric in $y$, i.e., $v(x,y)=v(x,y')$ for $|y|=|y'|=r$. Observe that the Laplacian of $v$ in terms of the variables $x$ and $r$ looks like the left hand side of \eqref{2.8},
$$
\Delta v=\Delta_x v+\frac{\tau}{r}v_r+v_{rr}.
$$
Thus, the function $v$ can be seen as the harmonic extension of $u$ from $\R^n$ to $\R^{n+1+\tau}$. The latter, of course, has no meaning when $\tau$ is not an integer, but as it turns out, solutions of \eqref{2.8} still carry many properties of harmonic functions when $\tau$ is not an integer. The fundamental solution of the Laplacian in $n+1+\tau$ dimension is, \cite[p. 22]{E10}, for $n-1+\tau>1$,
$$
\phi(x,y):=\frac{b_{n,s}}{|(x,y)|^{n-1+\tau}}=\frac{b_{n,s}}{\left(|x|^2+|y|^2\right)^\frac{{n-1+\tau}}{2}},
$$
where the constant $b_{n,s}$ is defined by
$$
b_{n,s}:=\frac{\Gamma\left(\frac{n}{2}-s\right)}{4^s\pi^{\frac{n}{2}}\Gamma(s)}.
$$ 
The function
\begin{equation}\label{fundamental}
	\phi(x,0):=\phi(x):=
	\begin{cases}
		\displaystyle\frac{b_{n,s}}{|x|^{n-2s}},& \textit{ if } n\ge2,\\ \ \\
		\displaystyle-\frac{1}{\pi}\log|x|, & \textit{ if } n=1,
	\end{cases}
\end{equation}
where $x\in\R^n\setminus\{0\}$ plays the role of the fundamental solution for the fractional Laplacian, i.e., it solves (in the distributional sense, \eqref{distribution}) the equation $(-\Delta)^s\phi=\delta_0$, where $\delta_0$ is the Dirac delta evaluated at zero, \cite[Theorem 2.3]{B16}. Observe also that as $v$ solves the problem \eqref{2.8}-\eqref{2.9}, it can be written (see \cite[p. 37]{E10}) explicitly in terms of the Poisson kernel for the half-space:
\begin{equation}\label{Poisson}
v(x,y)=\int_{\R^n}P(x-\xi,y)u(\xi)\,d\xi,
\end{equation}
where 
\begin{equation}\label{2.11}
	P(x,y):=B_{n,s}\frac{y^{2s}}{\left(|x|^2+|y|^2\right)^{\frac{n+2s}{2}}}.
\end{equation}
The kernel $P$ is indeed the Poisson kernel, since it solves \eqref{2.8} for $y>0$ and noting that $P(x,y)=y^{-n}P(x/y,1)$, converges, as $y\to0$, to a multiple of the Dirac delta. The constant $B_{n,s}$ is chosen such that
\begin{equation}\label{kernel}
	\int_{\R^n}P(x-\xi,y)\,d\xi=1.
\end{equation}

Finally, we bring another definition of the fractional Laplacian in terms of the extension function $v$. 

\begin{proposition}\label{p2.2}
	$(-\Delta)^su=-c_{n,s}\displaystyle\lim_{y\to0+}y^{1-2s}v_y$.
\end{proposition}
\begin{proof}
	Recalling \eqref{Poisson}, \eqref{kernel}, \eqref{2.9}, \eqref{2.11} and \eqref{1.1}, we compute
	\begin{equation*}
		\begin{split}		
		\lim_{y\to0^+}y^{1-2s} v_y&=\lim_{y\to0^+}\frac{v(x,y)-v(x,0)}{y^{2s}}\\
		&=\lim_{y\to0^+}\frac{1}{y^{2s}}\int_{\R^n}P(x-\xi,y)\left(u(\xi)-u(x)\right)\,d\xi\\
		&=\lim_{y\to0^+}\int_{\R^n}\frac{u(\xi)-u(x)}{\left(|x-\xi|^2+|y|^2\right)^{\frac{n+2s}{2}}}\,d\xi\\
		&=P.V.\int_{\R^n}\frac{u(\xi)-u(x)}{|x-\xi|^{n+2s}}\,d\xi\\
		&=-c_{n,s}^{-1}(-\Delta)^su(x).		
		\end{split}
	\end{equation*}
\end{proof}
Furthermore, a reflection argument makes sure that \eqref{2.10} makes sense in a ball of radius $r$ centered at $\{y=0\}$ in dimension $n+1$.
\begin{lemma}\label{l2.1}
	If $v:\R^n\times[0,+\infty)\to\R$ solves \eqref{2.8} such that for $|x|\le r$,
	\begin{equation}\label{extensionargument}
	\lim_{y\to0}y^{1-2s}v_y(x,y)=0,
	\end{equation}
	then
	\begin{equation}\label{reflexion}
		\tilde{v}(x,y):=
		\begin{cases}
			v(x,y), & y\ge0,\\
			v(x,-y), & y<0
		\end{cases}
	\end{equation}
is a weak solution of 
\begin{equation*}\label{reflectedequation}
	\Div\left(|y|^{1-2s}\nabla\tilde{v}\right)=0
\end{equation*}
in the $(n+1)$ dimensional ball of radius $r$.
\end{lemma}
\begin{proof}
	We need to verify that
	$$
	\int_{B_r^{n+1}}|y|^{1-2s}\nabla\tilde{v}\cdot\nabla\varphi\,dx\,dy=0,
	$$
	for any test function $\varphi\in C_0^\infty(B_r^{n+1})$, where $B_r^{n+1}:=\left\{(x,y);\,|x|^2+|y|^2<r^2\right\}$.
	Separating a strip of width $\varepsilon>0$ around $y=0$ in $B_r^{n+1}$, we write
	\begin{equation*}
		\begin{split}
			&\int_{B_r^{n+1}}|y|^{1-2s}\nabla\tilde{v}\cdot\nabla\varphi\,dx\,dy=\int_{B_r^{n+1}\setminus|y|<\varepsilon}+\int_{B_r^{n+1}\cap|y|<\varepsilon}\\
			&=\int_{B_r^{n+1}\setminus|y|<\varepsilon}\Div\left(|y|^{1-2s}\varphi\nabla\tilde{v}\right)\,dx\,dy+\int_{B_r^{n+1}\cap|y|<\varepsilon}|y|^{1-2s}\nabla\tilde{v}\cdot\nabla\varphi\,dx\,dy\\
			&=\int_{B_r^{n+1}\cap|y|=\varepsilon}\varphi|y|^{1-2s}\tilde{v}_y(x,\varepsilon)\,dx+\int_{B_r^{n+1}\cap|y|<\varepsilon}|y|^{1-2s}\nabla\tilde{v}\cdot\nabla\varphi\,dx\,dy.
		\end{split}
	\end{equation*}
The first integral in the right hand side of the above equality goes to zero, as $\varepsilon\to0$. So does the second integral, as $|y|^{1-2s}|\nabla v|^2$ is locally integrable.
\end{proof}
\begin{remark}\label{reflexremark}
	In fact (see Theorem \ref{regularitytheorem} below) \eqref{extensionargument} implies that $v$ is $C^\infty$ near $x$, and the limit in \eqref{extensionargument} is uniform. However, in general, we understand it in the weak sense. 	
\end{remark}

Proposition \ref{p2.2} and \eqref{2.10} show the importance of the extension argument. As it turns out, the study of a non-local operator (the fractional Laplacian) can be reduced to the study of a local operator in a higher dimensional space (as, for example, in \cite{CRS10,TT15}). This comes with the price of the weighted term $|y|^{1-2s}$ in the equation, but that weight belongs to the second Muchenhoupt class $A_2$, meaning
$$
\int_B|y|^{1-2s}\int_B|y|^{2s-1}<\infty,
$$
where $B$ is any ball in $\R^{n+1}$. Note also that this weight does not depend on the tangential variable, allowing to consider translations in $x$. These lead to Sobolev embeddings, Poincar\'e inequality, estimates of the Green function, etc., \cite{B16,FJK82,FKS82,R38}.

The extension argument reveals that a stochastic process with jumps in $\R^n$ can be seen as the ``trace" of a classical stochastic process in $\R^{n}\times[0,\infty)$ (a random walk with jumps in $\R^n$ can be interpreted as a classical random walk in $\R^{n+1}$). In other words, every time the classical stochastic process in $\R^n\times[0,\infty)$ hits $\R^n\times\{0\}$, it induces a jump process in $\R^n$.

\section{Elementary properties}\label{s3}
It is obvious that the fractional Laplacian is a linear operator, i.e.,
$$
(-\Delta)^s(u+v)=(-\Delta)^su+(-\Delta)^sv
$$
and
$$
(-\Delta)^s(cu)=c(-\Delta)^su,\quad c\in\R.
$$
It is noteworthy, that like the classical Laplacian, the fractional Laplacian is translation and rotation invariant, \cite[Lemma 2.7]{G19}. We bring here other elementary properties, such as homogeneity, asymptotics of the fractional Laplacian and the semi-group property. In fact, \eqref{1.1}, one easily checks that
\begin{equation*}\label{3.3}
	(-\Delta)^su(\lambda u)=\lambda^{2s}(-\Delta)^su.
\end{equation*}
The latter means that the fractional Laplacian is a homogeneous operator of order $2s$.
\begin{lemma}\label{l3.1}
	If $u\in\mathcal{S}$, then
	\begin{equation*}\label{3.1}
		\lim_{s\to0^+}(-\Delta)^su=u\,\,\,\textrm{ and }\,\,\,\lim_{s\to1^-}(-\Delta)^su=-\Delta u.
	\end{equation*}
\end{lemma}
\begin{proof}
	This follows from \eqref{Fourier}. Indeed, the case of $s=0$ is obvious. Also,
	\begin{equation*}
		\begin{split}
			-\Delta u(x)&=-\Delta\left(\mathcal{F}^{-1}(\hat{u})\right)(x)=-\Delta\left(\int_{\R^n}\hat{u}(\xi)e^{2\pi i\xi\cdot x}\,d\xi\right)\\
			&=\int_{\R^n}\left(2\pi|\xi|\right)^2\hat{u}(\xi)e^{2\pi i\xi\cdot x}\,d\xi=\mathcal{F}^{-1}\left((2\pi|\xi|)^2\hat{u}(\xi)\right).
		\end{split}
	\end{equation*}
\end{proof}
\begin{remark}\label{r3.1}
	Lemma \ref{l3.1} can also be deduced using Definition \ref{d2.1} and the assymptotics of constant $c_{n,s}$, 
	\begin{equation*}\label{3.2}
		\lim_{s\to0^+}\frac{c_{n,s}}{s(1-s)}=\frac{4n}{\omega_{n-1}}\,\,\,\textrm{ and }\,\,\,\lim_{s\to1^-}\frac{c_{n,s}}{s(1-s)}=\frac{2}{\omega_{n-1}},
	\end{equation*}
	where $\omega_{n-1}$ is the $(n-1)$-dimensional measure of the unit sphere. We refer the reader to \cite[Section 4]{DPV12}, where the calculations are carried out (see also \cite[Theorems 3 and 4]{S19} for the proof using definition \eqref{1.1}).
\end{remark}
The fractional Laplacian also enjoys the semi-group property, as states the following proposition.
\begin{proposition}\label{p3.1}
	If $u\in\mathcal{S}$, $s,t\in(0,1)$ and $s+t\le1$, then
	$$
	(-\Delta)^{s+t}u=(-\Delta)^{s}(-\Delta)^{t}u=(-\Delta)^{t}(-\Delta)^{s}u.
	$$
\end{proposition}
\begin{proof}
	This directly follows from \eqref{Fourier}. Indeed,
	\begin{equation*}
		\begin{split}
			\mathcal{F}\left((-\Delta)^{s+t}u\right)&=(2\pi|\xi|)^{2(s+t)}\hat{u}=(2\pi|\xi|)^{2s}(2\pi|\xi|^{2t})\hat{u}\\
			&=\mathcal{F}\left((-\Delta)^s(-\Delta)^tu\right)\\
			&=\mathcal{F}\left((-\Delta)^t(-\Delta)^su\right).
		\end{split}
	\end{equation*} 
	It now suffices to apply the inverse Fourier transform.
\end{proof}

\section{The $s$-mean value property}\label{s4}
Classical harmonic functions enjoy the mean value property: a value of a harmonic function at a point is equal to its average over spheres (or balls) centered at that point. The converse to the mean value property also is true: if at a given point $x$ a function is equal to its average over spheres centered at $x$, then it must be harmonic in a neighborhood of $x$. Similar principle is true for $s$-harmonic functions. The non-local nature of the fractional Laplacian, however, requires refinement of the argument. Once again we see that the spheres are replaced by the ``non-local boundary''. Namely, the value of an $s$-harmonic function at a point, is equal to its ``average'' defined by a convolution of the function with the $s$-mean kernel. More precisely, for $r>0$ set
\begin{equation*}
	A_r(y):=
	\begin{cases}
		\displaystyle a_{n,s}\frac{r^{2s}}{\left(|y|^2-r^2\right)^s|y|^n}, & y\in\R^n\setminus\overline{B}_r,\\
		0, & y\in\overline{B}_r,
	\end{cases}
\end{equation*}
where the constant $a_{n,s}$ is chosen such that 
\begin{equation}\label{meankernel}
\int_{\R^n\setminus B_r}A_r(y)\,dy=1.
\end{equation}
In fact, \cite[Section 15]{G19} we have
$$
a_{n,s}:=\frac{\sin(\pi s)\Gamma\left(\frac{n}{2}\right)}{\pi^{\frac{n}{2}+1}}.
$$
The following mean value property holds, \cite[Theorem 2.2]{B16}, \cite[Proposition 15.7]{G19}.
\begin{theorem}\label{meanvalue}
	Let $u\in L_s^1(\R^n)$ be $C^{2s+\varepsilon}$ in a neighborhood of $x\in\R^n$. If for any small $r>0$ one has
	\begin{equation}\label{MV}
		u(x)=\int_{\R^n\setminus B_r}A_r(y)u(x-y)\,dy,
	\end{equation}
	then $u$ is $s$-harmonic at $x$.
\end{theorem}
\begin{proof}
	From \eqref{meankernel} and \eqref{MV} we get
	\begin{equation*}		
			0=u(x)-\int_{\R^n\setminus B_r}A_r(y)u(x-y)\,dy=a_{n,s}r^{2s}
			\int_{\R^n\setminus B_r}\frac{u(x)-u(x-y)}{\left(|y|^2-r^2\right)^s|y|^n}\,dy,
	\end{equation*}
therefore,
\begin{equation*}\label{3.6}
	\int_{\R^n\setminus B_r}\frac{u(x)-u(x-y)}{\left(|y|^2-r^2\right)^s|y|^n}\,dy=0.
\end{equation*}
On the other hand, \eqref{2.1}, one has
$$
(-\Delta)^su(x)=\lim_{r\to0}\int_{\R^n\setminus B_r}\frac{u(x)-u(x-y)}{|y|^{n+2s}}\,dy,
$$
hence, it is enough to show that
\begin{equation}\label{3.7}
	\lim_{r\to0}\int_{\R^n\setminus B_r}\frac{u(x)-u(x-y)}{|y|^{n+2s}}\,dy=\lim_{r\to0}\int_{\R^n\setminus B_r}\frac{u(x)-u(x-y)}{\left(|y|^2-r^2\right)^s|y|^n}\,dy.
\end{equation}
To see this, take $R>\sqrt{2}r$ and split the integral, 
\begin{equation}\label{3.8}
	\int_{\R^n\setminus B_r}\frac{u(x)-u(x-y)}{\left(|y|^2-r^2\right)^s|y|^n}\,dy=\int_{\R^n\setminus B_{R}}+\int_{B_{R}\setminus B_r}:=I_r+J_r.
\end{equation}
For $I_r$ we have
\begin{equation}\label{3.9}
	\lim_{r\to0}I_r=\int_{\R^n\setminus B_{R}}\frac{u(x)-u(x-y)}{|y|^{n+2s}}\,dy.
\end{equation}
This is because when $y\in\R^n\setminus B_{R}$, one has
$$
\frac{|y|^2}{|y|^2-r^2}<2,
$$
therefore, as $u\in L_s^1(\R^n)$, 
$$
\int_{\R^n\setminus B_{R}}\frac{|u(x)-u(x-y)|}{\left(|y|^2-r^2\right)^s|y|^n}\,dy\le2^s\int_{\R^n\setminus B_{R}}\frac{|u(x)-u(x-y)|}{|y|^{n+2s}}\,dy<\infty,
$$
and we can use the dominated convergence theorem to pass to the limit, as $r\to0$ and obtain \eqref{3.9}. To pass to the limit in $J_r$, we notice that for $y\in B_{R}\setminus B_r$ and $s<1/2$ one has
$$
|u(x)-u(x-y)|\le c|y|^{2s+\varepsilon},
$$
since $u\in C^{2s+\varepsilon}$ in a neighborhood of $x$, where $c>0$ is a universal constant. For $s\ge1/2$ the $C^{2s+\varepsilon}=C^{1,2s+\varepsilon-1}$ regularity of $u$ in the same neighborhood provides
\begin{equation}\label{3.10}
	\begin{split}
		|u(x)-u(x-y)-y\cdot Du(x)|&=\left|\int_0^1y\left(Du(x-ty)-Du(x)\right)\,dt\right|\\
		&\le|y|\int_0^1|Du(x-ty)-Du(x)|\,dt\\
		&\le c|y|^{2s+\varepsilon}.
	\end{split}
\end{equation}
Observe that
$$
\int_{B_R\setminus B_r}\frac{y\cdot Du(x)}{\left(|y|^2-r^2\right)^s|y|^n}\,dy=\int_{B_R\setminus B_r}\frac{y\cdot Du(x)}{|y|^{n+2s}}\,dy=0,
$$
since we are integrating even functions over a symmetrical domain. Therefore, setting
\begin{equation*}
	\begin{split}
		H_r&:=J_r-\int_{B_R\setminus B_r}\frac{u(x)-u(x-y)}{|y|^{n+2s}}\,dy\\
		&=\int_{B_R\setminus B_r}(u(x)-u(x-y)-y\cdot Du(x))\left[\frac{1}{\left(|y|^2-r^2\right)^s|y|^n}-\frac{1}{|y|^{n+2s}}\right]\,dy.
	\end{split}
\end{equation*}
Using \eqref{3.10}, passing to polar coordinates and changing variables by $\rho=rt$, we estimate
\begin{equation*}
	\begin{split}
		|H_r|&\le c\int_{B_R\setminus B_r}|y|^{2s+\varepsilon}\left[\frac{1}{\left(|y|^2-r^2\right)^s|y|^n}-\frac{1}{|y|^{n+2s}}\right]\,dy\\
			&\le c\int_r^R\rho^{\varepsilon-1}\left[\frac{\rho^{2s}}{(\rho^2-r^2)^s}-1\right]\,d\rho\\
			&\le cr^\varepsilon\int_{1}^{\frac{R}{r}}t^{\varepsilon-1}\left[\frac{t^s}{(t-1)^s}-1\right]\,dt.
		\end{split}
	\end{equation*}
It remains to check that
\begin{equation}\label{3.11}
	\int_{1}^{\frac{R}{r}}t^{\varepsilon-1}\left[\frac{t^s}{(t-1)^s}-1\right]\,dt<\infty, 
\end{equation}
since combining it with the previous inequality we obtain
$$
\lim_{r\to0}H_r=0,
$$
or equivalently,
$$
\lim_{r\to0}J_r=\lim_{r\to0}\int_{B_R\setminus B_r}\frac{u(x)-u(x-y)}{|y|^{n+2s}}\,dy,
$$
which together with \eqref{3.8} and \eqref{3.9} gives \eqref{3.7}. To check \eqref{3.11}, we split the integral and notice that
$$
\int_1^{\sqrt{2}}t^{\varepsilon-1}\left[\frac{t^s}{(t-1)^s}-1\right]\,dt\le c\int_1^{\sqrt{2}}\left[\frac{1}{(t-1)^s}-\frac{1}{t^s}\right]\,dt<\infty
$$
and
\begin{equation*}
	\begin{split}
		\lim_{r\to0}\int_{\sqrt{2}}^{\frac{R}{r}}t^{\varepsilon-1}\left[\frac{t^s}{(t-1)^s}-1\right]\,dt&\le\int_{\sqrt{2}}^{\infty}t^{\varepsilon-1}\left[\left(1-\frac{1}{t}\right)^{-s}-1\right]\,dt\\
		&\le c\int_{\sqrt{2}}^\infty t^{\varepsilon-2}\,dt<\infty.
	\end{split}
\end{equation*}
\end{proof}
As a consequence of Theorem \ref{meanvalue}, one obtains a representation formula via Poisson kernel for the solution of the non-local Dirichlet problem on balls (just like in the local framework). For its proof we refer the reader to \cite[Theorem 2.10]{B16} (see also \cite[p. 17]{R38}, \cite[p. 122 and 112]{L72} and \cite[Theorem 15.2]{G19}). The fractional Poisson kernel is defined by
\begin{equation}\label{poissonkernel}
	P_r(x,y):=C_{n,s}\left(\frac{r^2-|x|^2}{|y|^2-r^2}\right)^s\frac{1}{|y-x|^n},
\end{equation}
where $r>0$ and
$$
C_{n,s}:=\frac{\sin(\pi s)\Gamma(\frac{n}{2})}{\pi^{\frac{n}{2}+1}}.
$$
The choice of the constant $C_{n,s}$ guarantees that $\displaystyle\int_{\R^n\setminus B_r}P_r(x,y)\,dy=1$.
\begin{theorem}\label{dirichletproblemhom}
	For $g\in L_s^1(\R^n)\cap C(\R^n)$ the unique solution of
	\begin{equation*} 
		\begin{cases}
			(-\Delta)^{s} u=0&\text{in }\,\,\,B_r,\\	
			u=g&\text{in }\,\,\,\R^n\setminus B_r
		\end{cases} 
	\end{equation*}
	is given by
	$$
	u(x)=\int_{\R^n\setminus B_r}g(y)P_r(x,y)\,dy,\quad x\in B_r.
	$$
\end{theorem}

\section{The maximum principle}\label{s5}
As it is well known, classical harmonic function in a bounded domain $\Omega\subset\R^n$ takes its extremal values on the boundary $\partial\Omega$. In other words, a non-negative harmonic function in $\Omega$ cannot vanish inside $\Omega$ (unless its identically zero). The literal analog of this for the fractional Laplacian fails: there exists a bounded fractional harmonic function $u$ that vanishes inside $\Omega$. One can construct such a function by defining $u$ outside $\Omega$ in a way that makes it feel the effect of far away data, \cite[Theorem 2.3.1]{BV16}. However, as remarked earlier, if we think of $\R^n\setminus\Omega$ as the ``non-local boundary'', several properties that the classical Laplacian enjoys remain true for the fractional Laplacian, including the maximum principle.
\begin{theorem}\label{maximumprinciple}
	If $(-\Delta)^su\ge0$ in $\Omega$ and $u\ge0$ in $\R^n\setminus\Omega$, then $u\ge0$ in $\Omega$. Moreover, $u>0$, unless $u\equiv0$.
\end{theorem}
\begin{proof}
	We divide the proof into two steps.\\	
	\textit{Step 1.} First we show that $u\ge0$ in $\Omega$. If not, then there exists a point in $\Omega$, where $u$ is strictly negative. Let $x_0\in\Omega$ be a point where $u$ takes its minimum. We have $u(x_0)<0$. Note that in fact $x_0$ is a global minimum, since $u\ge0$ outside of $\Omega$, and hence
	\begin{equation}\label{3.5}
		2u(x_0)-u(x_0+y)-u(x_0-y)\le0,\,\,\,\forall y\in\R^n.
	\end{equation}
	On the other hand, for $R>0$ large enough, if $y\in\R^n\setminus B_R$, then
	both $x_0+y$ and $x_0-y$ stay outside of $\Omega$, and hence,
	\begin{equation}\label{3.6}
		u(x_0+y)\ge0,\,\,\,\textrm{and}\,\,\,u(x_0-y)\ge0.
	\end{equation}
	Consequently, using \eqref{2.2}, \eqref{3.5} and \eqref{3.6}, we obtain
	\begin{equation*}
		\begin{split}
			0&\le\int_{\R^n}\frac{2u(x_0)-u(x_0+y)-u(x_0-y)}{|y|^{n+2s}}\,dy\\
			&\le\int_{\R^n\setminus B_R}\frac{2u(x_0)-u(x_0+y)-u(x_0-y)}{|y|^{n+2s}}\,dy\\
			&\le\int_{\R^n\setminus B_R}\frac{2u(x_0)}{|y|^{n+2s}}\,dy<0,
		\end{split}
	\end{equation*}
	a contradiction.\\
	\textit{Step 2.} We now show that the inequality is strict in $\Omega$, unless $u\equiv0$. If that is not the case, then there is $z\in\Omega$ such that $u(z)=0$. Step 1 provides $u(z+y)\ge0$ and $u(z-y)\ge0$ for any $y\in\R^n$. Since $(-\Delta)^su(z)\ge0$, \eqref{2.2} implies
	\begin{equation*}
		\begin{split}
			0&\le\int_{\R^n}\frac{2u(z)-u(z+y)-u(z-y)}{|y|^{n+2s}}\,dy\\
			&=-\int_{\R^n}\frac{u(z+y)+u(z-y)}{|y|^{n+2s}}\,dy\le0,
		\end{split}
	\end{equation*}
	which is possible only when $u\equiv0$.
\end{proof}
As a direct consequence, we get the comparison principle for the fractional Laplacian and uniqueness of the solution of the Dirichlet problem.
\begin{corollary}\label{c3.1}
	If $(-\Delta)^su\ge0$, $(-\Delta)^sv\ge0$ in $\Omega$ and $u\ge v$ outside of $\Omega$, then $u\ge v$ in the whole $\R^n$.
\end{corollary}
\begin{corollary}\label{c3.2}
	If $f\in C(\Omega)$ and $\varphi\in C(\R^n\setminus\Omega)$, then there is a unique $u\in L_1^s(\R^n)\cap C_{\loc}^{2s+\varepsilon}$ such that $(-\Delta)^su=f$ in $\Omega$ and $u=\varphi$ in $\R^n\setminus\Omega$.
\end{corollary}

We refer the interested reader to \cite[Proposition 4.1]{RO16} for a maximum principle for a general class of non-local operators.

\section{The Harnack inequality}\label{s6}
The Harnack principle for the classical Laplacian states that if a function $u$ is harmonic and non-negative in $B_r(x_0)$, then in a smaller ball its values are all comparable, i.e., there exists a constant $C>0$ independent of $u$, $x_0$ and $r>0$ such that
$$
\sup_{B_{\frac{r}{2}}(x_0)}u\le C\inf_{B_{\frac{r}{2}}(x_0)}u.
$$
Since the maximum principle for the fractional Laplacian needed a refinement, it is not surprising that the Harnack inequality also needs a refinement, as the classical Harnack inequality fails in the non-local framework, \cite[page 9]{AV19}, \cite[Lemma 2.1]{BC83}. It can be seen by constructing a counterexample using approximation of $w(x):=|x|^2$ by fractional harmonic functions (for the proof of this remarkable property see Theorem \ref{approximations} below). Namely, for $\varepsilon\in(0,\frac{1}{8})$ let $v_\varepsilon$ be fractional harmonic approximation of $w$ in $B_1$, i.e., $(-\Delta)^sv_\varepsilon=0$ in $B_1$ and
$$
\|w-v_\varepsilon\|_{C^2(B_1)}<\varepsilon.
$$
Then in $B_1\setminus B_{1/4}$ one has
$$
v_\varepsilon(x)\ge w(x)-\|w-v_\varepsilon\|_{L^\infty(B_1)}\ge\frac{1}{16}-\varepsilon>\varepsilon,
$$
but
$$
v_\varepsilon(0)\le w(0)+\|w-v_\varepsilon\|_{L^\infty(B_1)}<\varepsilon.
$$
Thus, $v_\varepsilon(0)<v_\varepsilon(x)$ in $B_1\setminus B_{1/4}$. Hence, 
$$
\inf_{B_1}v_\varepsilon=\inf_{\overline{B}_{1/4}}v_\varepsilon.
$$
Set now
$$
u_\varepsilon(x):=v_\varepsilon(x)-v_\varepsilon(y),\,\,\,x\in B_1,
$$
where $y\in\overline{B}_{1/4}$ is a point, where $v_\varepsilon$ reaches its infimum. By definition $u_\varepsilon$ is $s$-harmonic and non-negative in $B_1$. Moreover, $u_\varepsilon>0$ in $B_1\setminus B_{1/4}$ and still
$$
\inf_{B_{1/2}}u_\varepsilon=u_\varepsilon(y)=0.
$$
Therefore, the classical Harnack inequality fails in the non-local setting. However, a suitably refined Harnack inequality holds.
\begin{theorem}\label{Harnackinequality}
	If $u\in L^\infty(\R^n)\cap C^2(B_r(x_0))$ is $s$-harmonic in $B_r(x_0)$ and $u\ge0$ in $\R^n$, then there exists a constant $C>0$ independent of $u$, $x_0$ and $r>0$ such that
	$$
	\sup_{B_{\frac{r}{2}}(x_0)}u\le C\inf_{B_{\frac{r}{2}}(x_0)}u.
	$$
\end{theorem}
\begin{proof}
	This can be proved as in the classical case, using Theorem \ref{dirichletproblemhom}, \cite[Lemma 2.1]{BC83}. Another proof can be found in \cite[Proposition 2.3.4]{BV16}. The proof we bring here is much more compact and makes use of the extension argument, \cite[Theorem 5.1]{CS07}. 
	
	Let $v:\R^n\times[0,+\infty)$ be the solution of the extension problem \eqref{2.8}-\eqref{2.9}. Since $u\ge0$ in $\R^n$, recalling \eqref{Poisson}, also $v\ge0$. If $\tilde{v}$ is the reflection of $v(x,y)$ through the hyperplane $\{y=0\}$, \eqref{reflexion}, then as $u$ is fractional harmonic in $B_r(x_0)$, Lemma \ref{l2.1} yields
	$$
	\Div\left(|y|^{1-2s}\nabla\tilde{v}\right)=0
	$$
	in the $(n+1)$ dimensional ball of radius $r$ centered at $(x_0,0)$. We can apply the Harnack inequality for $\tilde{v}$, \cite[Lemma 2.3.5]{FKS82}, which gives the Harnack inequality for $u$.
\end{proof}

\section{Liouville theorem}\label{s7}
For the classical Laplacian the Liouville theorem states that entire harmonic functions that are bounded from below (or above) are constants. The same conclusion is true for entire fractional harmonic functions, \cite{BKN02}. In fact, entire $s$-harmonic functions are affine, and constant when $0<s\le1/2$, \cite[Theorem 1.3]{CDL15}, \cite[Theorem 1.1]{F16}. Here we bring a proof of a Liouville type theorem under weaker condition, obtained in \cite[Theorem 1.2]{CDL15}.
\begin{theorem}\label{Liouville}
	If $u\in L_s^1(\R^n)$ is $s$-harmonic and 
	\begin{equation}\label{liouvillecondition}
	\liminf_{|x|\to\infty}\frac{u(x)}{|x|^\gamma}\ge0
	\end{equation}
	for some $\gamma\in[0,1]$, $\gamma<2s$, then $u$ is a constant in $\R^n$.
\end{theorem}
\begin{proof}
	This follows from the fact that for $|x|<r$,
	\begin{equation}\label{3.14}
		u(x)=\int_{|y|>r}P_r(x,y)u(y)\,dy,
	\end{equation}
where $P_r$ is the fractional Poisson kernel for the ball $B_r$, defined by \eqref{poissonkernel}, Theorem \ref{dirichletproblemhom}. Notice that it is enough to show that for all unit vectors $\nu$ one has
\begin{equation}\label{nonegativederivative}
	D_\nu u\ge0.
\end{equation}
Indeed, since $\nu$ is arbitrary, then $Du=0$, hence $u$ is a constant in $\R^n$. To see \eqref{nonegativederivative}, using \eqref{3.14} we calculate
$$
D_iu(x)=-\int_{|y|>r}P_r(x,y)\left[\frac{2sx_i}{r^2-|x|^2}+\frac{n(x_i-y_i)}{|y-x|^2}\right]u(y)\,dy,
$$
therefore,
\begin{equation}\label{3.16}
	D_\nu u(x)=-\int_{|y|>r}P_r(x,y)\left[\frac{2sx\cdot\nu}{r^2-|x|^2}+\frac{n(x-y)\cdot\nu}{|y-x|^2}\right]u(y)\,dy.
\end{equation}
On the other hand, for any $\varepsilon>0$ fixed and $|y|$ sufficiently large, \eqref{liouvillecondition} implies
\begin{equation}\label{3.18}
	u(y)\ge-\varepsilon|y|^\gamma.
\end{equation}
For each fixed $x$, one can choose $r>0$ large enough to guarantee
\begin{equation}\label{3.19}
	\left|\frac{2sx\cdot\nu}{r^2-|x|^2}\right|\le\frac{1}{r},
\end{equation}
and for $|y|>r$ also
\begin{equation}\label{3.20}
	\left|\frac{n(x-y)\cdot\nu}{|y-x|^2}\right|\le\frac{n}{|y-x|}\le\frac{2n}{r}.
\end{equation}
Rewriting \eqref{3.16} as
\begin{equation*}
	\begin{split}
		D_\nu u(x)=&-\int_{|y|>r}P_r(x,y)\left[\frac{2sx\cdot\nu}{r^2-|x|^2}+\frac{n(x-y)\cdot\nu}{|y-x|^2}\right]\left[u(y)+\varepsilon|y|^\gamma\right]\,dy\\
		&+\int_{|y|>r}P_r(x,y)\left[\frac{2sx\cdot\nu}{r^2-|x|^2}+\frac{n(x-y)\cdot\nu}{|y-x|^2}\right]\varepsilon|y|^\gamma\,dy:=I+J,
	\end{split}
\end{equation*}
and using \eqref{3.18}, \eqref{3.19}, \eqref{3.20} and \eqref{3.14}, we have
\begin{equation}\label{3.21}
	\begin{split}
		I&\ge-\frac{2n+1}{r}\int_{|y|>r}P_r(x,y)\left[u(y)+\varepsilon|y|^\gamma\right]\,dy\\
		&=-\frac{2n+1}{r}u(x)-\frac{2n+1}{r}\varepsilon\int_{|y|>r}P_r(x,y)|y|^\gamma\,dy.
	\end{split}
\end{equation}
Clearly the first term in the right hand side of \eqref{3.21} goes to zero, as $r\to\infty$. We now aim to estimate the second term. Using definition of the Poisson kernel, \eqref{poissonkernel}, triangle inequality and changing variables (first $|y|:=\tau$ then $\tau:=rt$), for a constant $C>0$ we obtain
\begin{equation*}
	\begin{split}
		&\frac{2n+1}{r}\varepsilon\int_{|y|>r}P_r(x,y)|y|^\gamma\,dy\\
		&=\frac{C\varepsilon}{r}\left(r^2-|x|^2\right)^s\int_{|y|>r}\frac{|y|^\gamma}{\left(|y|^2-r^2\right)^s|y-x|^n}\,dy\\
		&\le\frac{C\varepsilon}{r}\left(r^2-|x|^2\right)^s\int_{|y|>r}\frac{|y|^\gamma}{\left(|y|^2-r^2\right)^s(|y|-|x|)^n}\,dy\\
		&=\frac{C\varepsilon}{r}\left(r^2-|x|^2\right)^s\int_{r}^{\infty}\frac{\tau^{\gamma+n-1}}{\left(\tau^2-r^2\right)^s(\tau-|x|)^n}\,d\tau\\
		&=\frac{C\varepsilon}{r^{2s-\gamma+1}}\left(r^2-|x|^2\right)^s\int_{1}^{\infty}\frac{t^{\gamma+n-1}}{\left(t^2-1\right)^s(t-\frac{|x|}{r})^n}\,dt\\
		&\le\frac{C\varepsilon}{r^{1-\gamma}}\int_{1}^{\infty}\frac{t^{\gamma+n-1}}{\left(t^2-1\right)^s(t-\frac{|x|}{r})^n}\,dt\le C\varepsilon.
	\end{split}
\end{equation*}
The last inequality is a consequence of the fact that the previous integral is convergent (since $\gamma$ is assumed to be less than $2s$), and $\gamma\le1$. Also,
$$
|J|\le\frac{2n+1}{r}\varepsilon\int_{|y|>r}P_r(x,y)|y|^\gamma\,dy\le C\varepsilon.
$$
Hence, for $r$ large enough, 
$$
D_\nu u(x)\ge-C\varepsilon.
$$
Letting $\varepsilon\to0$ in the last inequality, we deduce \eqref{nonegativederivative}.
\end{proof}
\begin{remark}\label{Liouvilleremark}
Theorem \ref{Liouville} has interesting consequences. If $(-\Delta)^su=P$ in $\R^n$ in the distributional sense, \eqref{distribution}, where $s\in(0,1)$ and $P$ is a polynomial, then $u$ is affine and $P=0$, \cite[Theorem 1.2]{F16}. Furthermore, if $p\ge1$ and $u\in L^p(\R^n)$ is fractional harmonic in the sense of distributions, then $u\equiv0$, \cite[Corollary 1.3]{F16}.
\end{remark}

\section{Regularity estimates}\label{s8}
The following regularity estimates are from \cite{S07}.
\begin{lemma}\label{l4.1}
	If $u\in C^{\alpha}(\R^n)$ for $\alpha\in(2s,1]$, then $(-\Delta)^su\in C^{\alpha-2s}(\R^n)$. Moreover,
	$$
	\left[(-\Delta)^su\right]_{C^{\alpha-2s}}\le C[u]_{C^{\alpha}},
	$$
	where $C>0$ is a constant depending only on $\alpha$, $s$ and $n$.
\end{lemma}
\begin{proof}
	We use \eqref{2.1} to compute
	\begin{equation*}
			\begin{split}
					\left|(-\Delta)^su(x)-(-\Delta)^su(y)\right|&=c_{n,s}\left|\int_{\R^n}\frac{u(x)-u(x+z)-u(y)+u(y+z)}{|z|^{n+2s}}\,dz\right|\\
					&\le I_1+I_2,
				\end{split}
		\end{equation*}
where $I_1$ is the integral over a ball of radius $r$, and $I_2$ is the integral over $\R^n\setminus B_r$. Since $|u(x)-u(x+z)|\le [u]_{C^{\alpha}}|z|^\alpha$ and $|u(y)-u(y+z)|\le [u]_{C^{\alpha}}|z|^\alpha$, we estimate
$$
I_1\le c_{n,s}\left|\int_{B_r}\frac{2[u]_{C^{\alpha}}|z|^\alpha}{|z|^{n+2s}}\,dz\right|\le C[u]_{C^{\alpha}}r^{\alpha-2s}.
$$
To estimate $I_2$, we make use of $|u(x+z)-u(y+z)|\le[u]_{C^{\alpha}}|x-y|^\alpha$ to write
$$
I_2\le c_{n,s}\left|\int_{\R^n\setminus B_r}\frac{2[u]_{C^{\alpha}}|x-y|^\alpha}{|z|^{n+2s}}\,dz\right|\le C[u]_{C^{\alpha}}r^{-2s}|x-y|^\alpha.
$$
Thus, taking $r=|x-y|$, we obtain
$$
\left|(-\Delta)^su(x)-(-\Delta)^su(y)\right|\le C[u]_{C^{\alpha}}|x-y|^{\alpha-2s}.
$$
\end{proof}
\begin{corollary}\label{c3.1}
	If $u\in C^{1,\alpha}(\R^n)$ for $\alpha\in(2s,1]$, then $(-\Delta)^su\in C^{1,\alpha-2s}(\R^n)$. Moreover,
	$$
	\left[(-\Delta)^su\right]_{C^{1,\alpha-2s}}\le C[u]_{C^{1,\alpha}},
	$$
	where $C>0$ is a constant depending only on $\alpha$, $s$ and $n$.
\end{corollary}
\begin{proof}
	This follows from Lemma \ref{l4.1} combined with the fact that the fractional Laplacian commutes with differentiation.
\end{proof}
\begin{lemma}\label{l4.2}
	If $u\in C^{1,\alpha}(\R^n)$ for $\alpha\in(0,2s)$, then $(-\Delta)^su\in C^{\alpha-2s+1}(\R^n)$. Moreover,
	$$
	[u]_{C^{\alpha-2s+1}}\le C[u]_{C^{1,\alpha}},
	$$
	where $C>0$ is a positive constant depending only on $\alpha$, $s$ and $n$.
\end{lemma}
\begin{proof}
	If $s<1/2$, we argue as in the proof of Lemma \ref{l4.1} to get
	$$
	\left|(-\Delta)^su(x)-(-\Delta)^su(y)\right|\le I_1+I_2,
	$$
	with the same $I_1$ and $I_2$ as in the proof of Lemma \ref{l4.1}. As $u\in C^{1,\alpha}(\R^n)$, we can estimate
	\begin{equation*}
			\begin{split}
					|u(x)-u(x+z)-u(y)+u(y+z)|&\le|(Du(x)-Du(y))\cdot z|+[u]_{C^{1,\alpha}}|z|^{1+\alpha}\\
					&=\left(|x-y|^\alpha|z|+|z|^{1+\alpha}\right)[u]_{C^{1,\alpha}}, 
				\end{split}
		\end{equation*}
therefore 
$$
I_1\le C\left(|x-y|^\alpha r^{1-2s}+r^{1+\alpha-2s}\right)[u]_{C^{1,\alpha}},
$$
and as before, taking $r=|x-y|$ gives the desired result.

For the case of $s\ge1/2$, using Proposition \ref{p3.1}, we can decompose 
$$
(-\Delta)^su=(-\Delta)^{s-1/2}(-\Delta)^{1/2}u
$$
and observe that 
$$
(-\Delta)^{1/2}u=\sum_{i=1}^nR_iD_i,
$$
where $R_i$ is the $i$-th Riesz transform (see, for example, \cite[Section 6]{G19}).
\end{proof}
An iteration of the last two lemmas leads to the following result.
\begin{lemma}\label{l4.3}
	If $u\in C^{k,\alpha}$ and $k+\alpha-2s$ is not an integer, then $(-\Delta)^su\in C^{\beta,\gamma}$, where $\beta$ is the integer part of $k+\alpha-2s$ and $\gamma=k+\alpha-2s-\beta$.
\end{lemma}
\begin{remark}\label{regularityremark}
Schauder type estimates hold for the fractional Laplacian. In fact, if $(-\Delta)^su\in C^\alpha(B_1)\cap C(\overline{B}_1)$, then $u\in C^{\alpha+2s}(B_{1/2})$, \cite[Theorem 1.2]{BK17}. Actually, a more general estimate holds,
\begin{equation*}
	\|u\|_{C^{\alpha+2s}(B_{1/2})}\le C\left[\|(-\Delta)^su\|_{C^\alpha(B_1)}+\|u\|_{L^\infty(B_1)}+\int_{\R^n\setminus B_1}\frac{u(y)}{|y|^{n+2s}}\,dy\right],
\end{equation*}
for any $\alpha\ge0$ such that $\alpha+2s$ is not an integer, as long as the terms in the right hand side are well defined.
\end{remark}

\section{Green's function for the ball}\label{s10}
As in the case of the classical Laplacian, the notions of fundamental solution and Poisson kernel allow one to define the Green function. For $x$, $z\in B_r$ and $x\neq z$, the Green function is defined in the following way,
$$
G(x,z):=\phi(x-z)-\int_{\R^n\setminus B_r}\phi(z-y)P_r(y,x)\,dy,
$$
where $\phi$ is the fundamental solution defined by \eqref{fundamental}, and $P_r$ is the Poisson kernel defined by \eqref{poissonkernel}. It can be displayed in a more explicit way, \cite[Theorem 3.1]{B16},
\begin{equation}\label{Greenfunction}
G(x,z)=\kappa_{n,s}|z-x|^{2s-n}\int_0^{R(x,z)}\frac{t^{s-1}}{(t+1)^{\frac{n}{2}}}\,dt,
\end{equation}
where
$$
R(x,y):=\frac{(r^2-|x|^2)(r^2-|z|^2)}{r^2|x-z|^2},\,\,\,\textit{ if }\,\,\,n\ge2,
$$
and
\begin{equation}\label{Greenfunction1}
G(x,z)=\frac{1}{\pi}\log\left(\frac{r^2-xz+\sqrt{(r^2-x^2)(r^2-z^2)}}{r|z-x|}\right),\,\,\,\textit{ if }\,\,\,n=1,
\end{equation}
with
$$
\kappa_{n,s}:=\frac{\Gamma\left(\frac{n}{2}\right)}{4^s\pi^{\frac{n}{2}}\Gamma^2(s)}.
$$
The proof can be found in the celebrated work of Riesz, \cite{R38} (see also \cite{BGR61,G19} and \cite[Theorem 3.2]{B16}).
\begin{theorem}\label{dirichletproblemnonhom}
	If $f\in C^{2s+\varepsilon}(B_r)\cap C(\overline{B}_r)$, then the unique solution of
	\begin{equation*} 
		\begin{cases}
			(-\Delta)^{s} u=f&\text{in }\,\,\,B_r,\\	
			u=0&\text{in }\,\,\,\R^n\setminus B_r
		\end{cases} 
	\end{equation*}
	is given explicitly in terms of the Green function by
	$$
	u(x)=\int_{\R^n\setminus B_r}G(x,y)f(y)\,dy,\quad x\in B_r.
	$$
\end{theorem}
As a consequence of Theorems \ref{dirichletproblemhom} and \ref{dirichletproblemnonhom} we obtain an explicit representation of the solution of the Dirichlet problem in the ball of radius $r>0$.
\begin{theorem}\label{dirichletproblem}
	If $f\in C^{2s+\varepsilon}(B_r)\cap C(\overline{B}_r)$ and $g\in L_s^1(\R^n)\cap C(\R^n)$, then the unique solution of the problem 
	\begin{equation*} 
		\begin{cases}
			(-\Delta)^{s} u=f&\text{in }\,\,\,B_r,\\	
			u=g&\text{in }\,\,\,\R^n\setminus B_r
		\end{cases} 
	\end{equation*}
is given by
$$
u(x)=\int_{\R^n\setminus B_r}g(y)P_r(x,y)\,dy+\int_{\R^n\setminus B_r}G(x,y)f(y)\,dy,\,\,x\in B_r,
$$
where $P_r$ and $G$ are defined by \eqref{poissonkernel} and \eqref{Greenfunction}-\eqref{Greenfunction1} respectively.
\end{theorem}

\section{Fractional harmonic functions are $C^\infty$}\label{s11}
As it is well known, classical harmonic functions are $C^\infty$. Fractional harmonic functions enjoy the same regularity. This may seem like an obvious observation, as in several definitions above the fractional Laplacian was defined for $C^\infty$ functions, but in fact there is no loss of generality. Namely, even if one starts with the ``weakest'' regularity assumptions, fractional harmonic functions turn out to be $C^\infty$, \cite[Theorem 2.10]{BK17}, \cite[Corollary 1]{S19}.
\begin{theorem}\label{regularitytheorem}
	If $u\in L^\infty(\R^n)\cap C(\R^n\setminus B_r)$ is such that $(-\Delta)^su=0$ in $B_r$, $r>0$, then for any multi-index $\alpha\in\mathbb{N}_0^n$
	$$
	\|D^\alpha u\|_{L^\infty(B_{r/2})}\le Cr^{-|\alpha|}\|u\|_{L^\infty(\R^n\setminus B_r)},
	$$
	where $C>0$ is a constant depending only on $n$, $s$ and $\alpha$.
\end{theorem}
\begin{proof}
	This follows from the smoothness of the Poisson kernel \eqref{poissonkernel}. Observe that without loss of generality we may assume $r=1$. Indeed, if 
	\begin{equation}\label{4.4}
	\|D^\alpha u\|_{L^\infty(B_{1/2})}\le C\|u\|_{L^\infty(\R^n\setminus B_1)},
	\end{equation}
	then by rescaling $y:=rx$, $v(y):=u(x)$, $x\in B_1$, one has $D^\alpha u(x)=r^{|\alpha|}|D^\alpha v(y)|$, which yields,
	\begin{equation*}		
			r^{|\alpha|}|D^\alpha v(y)|=|D^\alpha u(x)|\le C\|u\|_{L^\infty(\R^n\setminus B_1)}=C\|v\|_{L^\infty(\R^n\setminus B_r)},
	\end{equation*}	
	and the result follows. To prove \eqref{4.4}, note that from \eqref{poissonkernel} and Theorem \ref{dirichletproblemhom}, we have
	$$
	u(x)=\int_{\R^n\setminus B_1}u(y)P_1(x,y)\,dy= C_{n,s}\int_{\R^n\setminus B_1}u(y)\left(\frac{1-|x|^2}{|y|^2-1}\right)^s\frac{dy}{|y-x|^n}.
	$$
	Hence
	\begin{equation*}
		\begin{split}
	D_iu(x)=&2s\int_{\R^n\setminus B_1}\frac{u(y)}{(|y|^2-1)^s}\frac{x_i(1-|x|^2)^{s-1}}{|x-y|^n}\,dy\\
	&-\int_{\R^n\setminus B_1}\frac{u(y)}{(|y|^2-1)^s}\frac{n(1-|x|^2)^s(x_i-y_i)}{|x-y|^{n+2}}\,dy,
		\end{split}
	\end{equation*}
therefore,
\begin{equation}\label{4.5}
|Du(x)|\le C\int_{\R^n\setminus B_1}\frac{|u(y)|}{(|y|^2-1)^s}\left[\frac{|x|(1-|x|^2)^{s-1}}{|x-y|^n}+\frac{(1-|x|^2)^s}{|x-y|^{n+1}}\right]\,dy,
\end{equation}
where $C>0$ is a constant depending on $s$ and $n$. On the other hand, if $|x|\le\frac{1}{2}$, then 
$$
\frac{3}{4}\le1-|x|^2\le1\,\,\,\textrm{ and }\,\,\,|x-y|\ge\frac{|y|}{2},
$$
which combined with \eqref{4.5} and passing to polar coordinates, yields
\begin{equation*}
	\begin{split}
|Du(x)|&\le C\|u\|_{L^\infty(\R^n\setminus B_1)}\int_{\R^n\setminus B_1}\left[\frac{1}{(|y|-1)^s|y|^n}+\frac{1}{(|y|-1)^s|y|^{n+1}}\right]\,dy\\
	  &\le C\|u\|_{L^\infty(\R^n\setminus B_1)}\int_1^\infty\left[\frac{1}{(\rho-1)^s\rho}+\frac{1}{(\rho-1)^s\rho^2}\right]\,d\rho\\
	  &\le C\|u\|_{L^\infty(\R^n\setminus B_1)}.
	\end{split}
\end{equation*}
Thus,
$$
|Du(x)|\le C\|u\|_{L^\infty(\R^n\setminus B_1)},\,\,\,\forall x\in B_{1/2}. 
$$
Reiterating the computation, we get \eqref{4.4} for any multi-index $\alpha$.
\end{proof}

\section{Density of fractional harmonic functions}\label{s12}
As we have seen above, fractional harmonic functions share lots of properties with classical harmonic functions. Obviously, there are also several significant differences that come from the non-local nature of the fractional Laplacian. In particular, as we will see below, any given smooth function can be locally approximated by fractional harmonic functions. This striking property, obtained in \cite{DSV17}, shows how faraway oscillations of a fractional harmonic function affect on its local behavior. In other words, fractional harmonic functions are dense in the set of locally smooth functions. There is no local counterpart of this property. Indeed, classical harmonic functions cannot have a strict local maximum, hence, functions with strict local maximum cannot be approximated by harmonic functions. It is noteworthy, that although this is purely non-local phenomenon, but a similar result does not hold for any non-local operator.
\begin{theorem}\label{approximations}
	If  $f\in C^k(\overline{B}_1)$, for $k\in\mathbb{N}$, then for any $\varepsilon>0$, there exists $R>0$ and $u\in H^s(\R^n)\cap C^s(\R^n)$ such that $u$ is fractional $s$-harmonic in $B_1$, vanishes outside of $B_R$
	$$
	\|f-u\|_{C^k(\overline{B}_1)}<\varepsilon.
	$$
\end{theorem}
\begin{proof}
	We sketch the proof in the one-dimensional case, as in \cite[Section 2.5]{BV16}. For the general proof we refer the reader to \cite[Theorem 1.1]{DSV17}. Notice that it is enough to prove the result for monomials. Indeed, by Stone-Weierstrass Theorem, for any $\varepsilon>0$ and a given $f\in C([0,1])$, there exists a polynomial $P$ such that
	$$
	\|f-P\|_{C^k(\overline{B}_1)}<\varepsilon.
	$$
	Combined with the linearity of the fractional Laplacian, this implies that it is enough to prove the theorem for monomials, i.e., it is enough to show that $P(x)=x^m$, $m\ge1$ can be approximated by an $s$-harmonic function $u_m$. In turn, to prove the latter, it is enough to show that for any $m\in\mathbb{N}$, there exist $R>r>0$, $x\in\R$ and $u$ such that 
	\begin{equation}\label{6.1}
		\begin{cases}
			(-\Delta)^su=0 & \textit{ in }\,\,(x-r,x+r),\\
			u=0 & \textit{ in }\,\,\R\setminus(x-R,x+R),
		\end{cases}
	\end{equation}	
	 and 
	 \begin{equation}\label{6.2}
	 	D^iu(x)=0,\,\,i\in\{0, 1,\ldots,m-1\},\,\,D^mu(x)=1. 
	 \end{equation}
 	Indeed, it implies that, up to a translation, $u(x)=x^m+O(x^{m+1})$ near the origin, hence, its blow-up
 	$$
 	u_\lambda(x):=\frac{u(\lambda x)}{\lambda^m}=x^m+\lambda O(x^{m+1}),
 	$$
 	being an $s$-harmonic function, for $\lambda$ small is arbitrarily close to $x^m$, which, as stated earlier, provides the desired result. Thus, it remains to makes sure there exists a function $u$ satisfying \eqref{6.1} and \eqref{6.2}. To that aim, let $\mathbb{L}$ be the set of all pairs $(u,x)$ satisfying \eqref{6.1}. Define the vector space
	 $$
	 V:=\left\{\left(u(x),Du(x),\ldots,D^mu(x)\right),\,\textit{ for }\,(u,x)\in\mathbb{L}\right\}.
	 $$
	 Directly can be verified that $V$ is a linear spaces. Moreover,
	 \begin{equation}\label{6.3}
	 	V=\R^{m+1}.
	 \end{equation}
 	Assume for a moment, that \eqref{6.3} is verified. As $(0,\ldots\,0,1)\in\R^{m+1}=V$, the pair $(u,x)$ satisfies \eqref{6.1} and \eqref{6.2}. Thus, we are left to prove \eqref{6.3}. We argue by contradiction and assume that \eqref{6.3} fails. Since $V$ is a linear space, it has to be a proper subspace of $\R^{m+1}$ and so it lies in a hyperplane. Consequently, there exists $c=(c_0,c_1,\ldots,c_m)\in\R^{m+1}\setminus\{0\}$ such that
 	$$
 	V\subseteq\left\{\mu\in\R^{m+1};\,\,c\cdot\mu=0\right\}.
 	$$
 	This means that the vector $c$ is orthogonal to any vector in $V$, i.e.,
 	\begin{equation}\label{6.4}
 	\sum_{i\le m}D^iu(x)=0.
 	\end{equation}
 	If $u(x)=x_+^s$, then $D^iu(x)=s(s-1)\ldots(s-i+1)x^{s-i}$, and multiplying with $x^{m-s}$, $x\neq0$, from \eqref{6.4} we get
 	$$
 	\sum_{i\le m}c_is(s-1)\ldots(s-i+1)x^{m-i}=0,
 	$$
 	i.e., $c_i=0$ for each $i$, or equivalently $c=0$, which is a contradiction. This completes the proof. Strictly speaking the function $x^s_+$, being $s$-harmonic, Theorem \ref{t5.1}, does not satisfy \eqref{6.1}, because it does not have a compact support. So to deduce the contradiction, one should assume that $u$ is a fractional harmonic function with compact support, which behaves like $x^s$ near the origin, and apply \eqref{6.4} for $x>0$ small.
\end{proof}

\medskip

\textbf{Acknowledgments.} The author was partially supported by the King Abdullah University of Science and Technology (KAUST) and by the Centre for Mathematics of the University of Coimbra (UIDB/00324/2020, funded by the Portuguese Government through FCT/MCTES).


\begin{thebibliography}{99}
	
	\bibitem{AV19} N. Abatangelo and E. Valdinoci, \textit{Getting acquainted with the fractional Laplacian}, Contemporary research in elliptic PDEs and related topics, 1-105, Springer INdAM Ser. 33, Springer, Cham, 2019.
	
	\bibitem{A09} D. Applebaum, \textit{L\'evy processes and stochastic calculus}, Second Edition, Cambridge Studies in Advanced Mathematics 116, Cambridge University
	Press, Cambridge, UK, 2009.
	
	\bibitem{BC83} R.F. Bass and M. Cranston, \textit{Exit times for symmetric stable srocesses in $\R^n$}, Ann. Probab. 11 (1983), 578-588.
	
	\bibitem{BGR61} R.M. Blumenthal, R.K. Getoor and D.B. Ray, \textit{On the distribution of first hits for the symmetric stable processes}, Trans. Amer. Math. Soc. 99 (1961), 540-554.
	
	\bibitem{BKN02} K. Bogdan, T. Kulczycki and A. Nowak, \textit{Gradient estimates for harmonic and $q$-harmonic functions of symmetric stable processes}, Illinois J. Math. 46 (2002), 541-556.
	
	\bibitem{B16} C. Bucur, \textit{Some observations on the Green function for the ball in the fractional Laplace framework}, Commun. Pure Appl. Anal. 15 (2016), 657-699.
	
	\bibitem{BK17} C. Bucur and A. Karakhanyan, \textit{Potential theoretic approach to Schauder estimates for the fractional Laplacian}, Proc. Amer. Math. Soc. 145 (2017), 637-651.
	
	\bibitem{BV16} C. Bucur and E. Valdinoci, \textit{Nonlocal diffusion and applications}, Lecture Notes of the Unione Matematica Italiana 20, Bologna, Springer, 2016.
	
	\bibitem{CRS10} L. Caffarelli, J.M. Roquejoffre and Y. Sire, \textit{Variational problems with free boundaries for the fractional Laplacian}, J. Eur. Math. Soc. 12 (2010), 1151-1179.
	
	\bibitem{CS07} L. Caffarelli and L. Silvestre, \textit{An extension problem related to the fractional Laplacian}, Comm. Partial Differential Equations 32 (2007), 1245-1260.	
	
	\bibitem{CS09} L. Caffarelli and L. Silvestre, \textit{Regularity theory for fully nonlinear integro-differential equations}, Comm. Pure Appl. Math. 62 (2009), 597-638.
	
	\bibitem{CTU20} L. Caffarelli, R. Teymurazyan and J.M. Urbano, \textit{Fully nonlinear integro-differential equations with deforming kernels}, Comm. Partial Differential Equations 45 (2020), 847-871.
	
	\bibitem{CDL15} W. Chen, L. D'Ambrosio and Y. Li, \textit{Some Liouville theorems for the fractional Laplacian}, Nonlinear Anal. 121 (2015), 370-381.
	
	\bibitem{DPV12} E. Di Nezza, G. Palatucci and E. Valdinoci, \textit{Hitchhiker's guide to the fractional Sobolev spaces}, Bull. Sci. Math. 136 (2012), 521-573.
	
	\bibitem{DSV17} S. Dipierro, O. Savin and E. Vandinoci, \textit{All functions are locally $s$-harmonic up to a small error}, J. Eur. Math. Soc. 19 (2017), 957-966.
	
	\bibitem{E10} L.C. Evans, \textit{Partial differential equations}, Second Edition, Graduate Studies in Mathematics 19, American Mathematical Society, Providence, RI, 2010.
	
	\bibitem{FJK82} E.B. Fabes, D. Jerison and C.E. Kenig, \textit{The Wiener test for degenerate elliptic equations}, Ann. Inst. Fourier (Grenoble) 32 (1982), 151-182.
	
	\bibitem{FKS82} E.B. Fabes, C.E. Kenig and R.P. Serapioni, \textit{The local regularity of solutions of degenerate elliptic equations}, Comm. Partial Differential Equations 7 (1982), 77-116.
	
	\bibitem{F16} M.M. Fall, \textit{Entire $s$-harmonic functions are affine}, Proc. Amer. Math. Soc. 144 (2016), 2587-2592.
		
	\bibitem{G19} N. Garofalo, \textit{Fractional thoughts}, New developments in the analysis of nonlocal operators, 1-135, Contemp. Math. 723, Amer. Math. Soc., Providence, RI, 2019.
	
	\bibitem{K17} M. Kwa\'snicki, \textit{Ten equivalent definitions of the fractional Laplace operator}, Fract. Calc. Appl. Anal. 20 (2017), 7-51.	
	
	\bibitem{L72} N.S. Landkof, \textit{Foundations of modern potential theory}, Springer-Verlag, New York-Heidelberg, 1972 (translated from Russian).
	
	\bibitem{MO69} S.A. Molchanov and E. Ostrovskii, \textit{Symmetric stable processes as traces of degenerate diffusion processes}, Teor. Verojatnost. i Primenen. 14 (1969), 127-130.
	
	\bibitem{KS91} I. Karatzas and S.E. Shreve, \textit{Brownian motion and stochastic calculus}, Second edition, Graduate Texts in Mathematics 113, Springer-Verlag, New York, 1991. xxiv+470 pp.
	
	\bibitem{R38} M. Riesz, \textit{Int\'egrales de Riemann-Liouville et potentiels}, Acta Sci. Math. Szeged 9 (1938), 1-42.
	
	\bibitem{RO16} X. Ros-Oton, \textit{Nonlocal elliptic equations in bounded domains: a survey}, Publ. Mat. 60 (2016), 3-26.
	
	\bibitem{S07} L. Silvestre, \textit{Regularity of the obstacle problem for a fractional power of the Laplace operator}, Comm. Pure Appl. Math. 60 (2007), 67-112.
	
	\bibitem{S19} P.R. Stinga, \textit{User's guide to the fractional Laplacian and the method of semigroups}, Handbook of fractional calculus with applications 2, 235-265, De Gruyter, 2019.
	
	\bibitem{TT15} E. Teixeira and R. Teymurazyan, \textit{Optimal design problems with fractional diffusions}, J. Lond. Math. Soc. (2) 92 (2015), 338-352.
	
\end{thebibliography}
\end{document}